\documentclass[english]{smfartss}

\usepackage{smfthm-treplo}
\usepackage{amssymb}
\usepackage{graphicx}
\usepackage[matrix, arrow,all,cmtip,color]{xy}
\usepackage{xcolor}
\usepackage{MnSymbol}
\usepackage{hyperref}
\usepackage{wrapfig}
\def\rrt#1#2#3#4#5#6#7{\xymatrix{ {#1} \ar[r]^{} \ar@{-\ge}[d]_{#2} & {#4} \ar[d]^{#5} \\ {#3}  \ar[r] \ar@{--\ge}[ur]^{#7}& {#6} }}

\def\rtt{\rightthreetimes}

\def\lra{\longrightarrow}
\def\lr{{lr}}
\def\lrl{{l}}

\def\rlr{{r}}

\def\id{\operatorname{id}}
\def\oo{\infty}

\def\ra{\searrow}

\def\llrra{\leftrightarrow}

\def\xra{\xrightarrow}

\title{A suggestion towards a finitist's realisation of 
topology}
\author[]{\tiny This kind of universality is what, we believe, turns the hidden wheels of the human thinking machinery.}
\email{mi\!\!\!ishap\!\!\!p@ma.sdf.org kip302002@yaho\!\!\!oo\!\!\!o.com \url{https://t.me/McVlr}}
\address{Konstantin Pimenov (SpbGU) \,\&\, Masha Gavrilovich (IPRERAN)}
\begin{document}\enlargethispage{7\baselineskip}
\begin{abstract}
We observe that the notion of a trivial Serre fibration, a Serre fibration, and being contractible, for finite CW complexes, 
can be defined in terms of the Quillen lifting property with respect to a single
 map  $M\to\Lambda$ of finite topological spaces (preorders) of size 5 and 3.
In particular, we observe that the 
double Quillen orthogonal $\{M\to\Lambda\}^{lr}$ is precisely the class of trivial Serre fibrations
if calculated in a certain category of nice topological spaces.
This suggests a question whether there is a finitistic/combinatorial definition of a model structure 
on the category of topological 
spaces entirely in terms of the single morphism $M\to\Lambda$,
apparently related to  the Michael continuous selection theory.
%
%
\end{abstract}  
\maketitle
\begin{wrapfigure}{r}{0.158\textwidth}\setcounter{figure}{1}
\includegraphics[scale=.26]{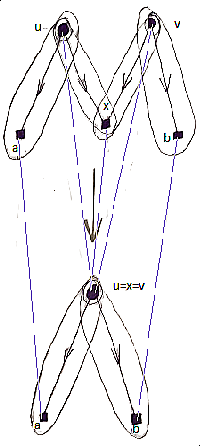}
\tiny Fig.1.{\tiny 
	 Arrow~$u\to a$ means $a\in cl\,u$.  
	 \begin{center} $\left\{\underset{a}{}{\swarrow} \overset{\raisebox{1pt}{$u$}}{}{\searrow} \underset{x}{}{\swarrow}\overset{\raisebox{1pt}{$v$}}{}{\searrow}\underset{b}{}\right\}$\hskip2pt 
	 \\\hskip2pt$\color{blue}\Downarrow$\hskip2pt
	 \\\hskip2pt$\left\{\underset{a}{}{\swarrow}\overset{\raisebox{1pt}{$u\!=\!x\!=\!v$}}{}{\searrow} \underset{b}{}\right\}$\end{center}
	}
\end{wrapfigure}
\section{Introduction} Being contractible, compact (for nice spaces), trivial Serre fibration (for nice spaces, with caveats), 
connected, dense,  extremally disconnected, 
zero-dimensional, and separation axioms $T_0,T_1$, $T_4,T_5$, can each be defined 
in terms of the Quillen lifting property \cite{1} 
and a single map of topological spaces (preorders), usually with less than 7 points \cite{2}.
This suggests a combinatorial, computational notation for these topological properties,
which could perhaps be of use in computer algebra and proof verification. 
This notation shows there is {\em finite combinatorics implicit in the basic definitions of topology}---what does it tell us ?  

In this note we show the finite combinatorics implicit in the basic definitions of {\em contractible}, {\em trivial fibrations}, and {\em fibrations}.  
We observe that for a certain map $M\to\Lambda$ of finite topological spaces (see Fig.~1), 
the double Quillen orthogonal (negation) $\left\{M\to \Lambda\right\}^{lr}$, defined below, is 
exactly the class of trivial Serre fibrations when calculated in a certain category of nice spaces.
If we calculate the same orthogonal in the category of (all) topological spaces, 
we only prove that 
\begin{itemize}\item[] 
$\text{a finite CW complex }X\text{ is contractible iff }
X\to \{o\}\,\,\in\,\,\left\{M\to \Lambda\right\}^{lr}$
\end{itemize}
 \enlargethispage{6\baselineskip}
 In fact, the precise choice of the map $M\to\Lambda$ 
 in the double Quillen orthogonal (negation) $\left\{M\to \Lambda\right\}^{lr}$\footnote{Recall that a morphism $i$ in a category has the {\em left lifting property} with respect to a morphism $p$, 
and $p$ also has the {\em right lifting property} with respect to $i$, 
denoted $i\rtt p$, 
iff for each $f:A\to X$ and $g:B\to Y$ such that $p\circ f = g \circ i$ there exists $h:B\to X$ such that $h\circ i = f$ and $p\circ h = g$.


For a class $P$ of morphisms in a category, its {\em left orthogonal} $P^{\rtt l}$ with respect to the lifting property, respectively its {\em right orthogonal} $P^{\rtt r}$, is the class of all morphisms which have the left, respectively right, lifting property with respect to each morphism in the class $P$. In notation,
$$
P^{\rtt l} := \{ i \,\,:\,\, \forall p\in P\,\, i\rtt p\},
P^{\rtt r} := \{ p \,\,:\,\, \forall i\in P\,\, i\rtt p\}, P^\lr:=(P^\lrl)^\rlr,..$$

Taking the orthogonal of a class $P$ is a simple way to define a class of morphisms excluding non-isomorphisms 
from $P$, in a way which is useful in a diagram chasing computation, and is often used to
define properties of morphisms starting from an explicitly given class of (counter)examples. 
  For this reason, 
it is convenient and intuitive to refer to $P^\lrl$ and $P^\rlr$ as {\em left, resp.~right, Quillen negation} of property $P$. 
See \cite{1} for a quick explanation and some examples.
}  is a way to add 
precise ``niceness'' assumptions to the ``na\i ve'' lifting property defining fibrations:
\begin{itemize}

\item[(wf)] a map $p:Y\to B$ is a trivial fibration 
iff the lifting property  
$A\to X \rtt  Y\xra p B$
holds whenever $A\subset X$ is a ``nice'' closed subset of a ``nice'' space $X$. 

\item[(f)] 
a map $p:Y\to B$ is a fibration  
iff, whenever $A\subset X$ closed and ``nice'', for any lifting problem  
$A\to X \rtt  Y\xra p B$, 
there exists a diagonal lifting  defined on some open neighbourhood of $A$.\footnote{
Formally in notation, for any commutative square
$$\xymatrix{ A \ar[r]^{f} \ar@{->}[d]_{i} & {Y} \ar[d]^{p} \\ {X}  \ar[r]^{\phi} & {B} }$$
there is an open $A\subset U\subset X$ and a map $\tilde f_U:U\to Y$ such that the diagram 
$$\xymatrix{ {} & A \ar[r]^{f} \ar@{..>}[dl] \ar@{->}[d]_{} & {Y} \ar[d]^{p} \\ 
{U}\ar@{..>}[r] \ar@{..>}[urr]|{\tilde f_U} & X  \ar[r]|{\phi_{|U}} & {B} }$$
commutes.
The diagram chasing rendering of this uses the non-Hausdorff mapping cone of $Y\xra p B$. 
}
\end{itemize}
We use the word ``nice'', in this paper, 
to mean various precise assumptions of the kind made to avoid 
spurious difficulties related to wild phenomena such as curves cheerfully filling cubes,
which are irrelevant from the point of view of the topological intuition of shapes,
cf.~\cite[\S5, pp.28/29]{4}.

The definition of Serre fibration chooses the nicest possible $A\subset X$ -- the inclusions of a sphere 
as the boundary of a ball. Michael continuous selection theory \cite[Thm.1.2]{6} chooses 
least(?) nice ones: an arbitrary closed subset of a Hausdorff paracompact space of finite Lebesgue dimension
(see \S\ref{s.3.1}, esp.~Thm.\ref{MTh1.2}, for a summary of \cite[Thm.1.2]{6} of Michael continuous selection theory; 
also see~Lemma~\ref{s.2.2.1}(4), \S\ref{s2.3}(ii), and Conjecture~\ref{s.2.4.2}).

As noted above, the map $M\to \Lambda$ does capture the implicit combinatorics of the definition of 
a trivial fibration in presence of the right ``niceness'' assumptions, i.e.~if calculated in a certain
subcategory of nice spaces, but it is not clear to us whether this implicit combinatorics is 
sufficient if calculated in the category of all topological spaces. Perhaps the reader would see this right away. 

It is easy to see that the map $M\to \Lambda$ captures the ``combinatorics'' implicit in the definition of normality:
a space $X$ is normal ($T_4$ but not necessarily $T_1$) iff $\emptyset\to X \,\rtt\, M\to \Lambda$:
indeed, to give a map $X\to \Lambda$ is to give two disjoint closed subsets of $X$ (the preimages 
of the two closed points of $\Lambda$), and to give a factorisation $X\to \Lambda$ is to give 
their disjoint neighbourhoods (the preimages of the open subsets of $M$ separating the preimages 
of the two closed points of $\Lambda$). Instead of $M\to\Lambda$, one may consider
the more complicated map implicit in the definition of hereditary normal (separation axiom $T_5$),
see the proof of Lemma~\ref{s.2.2.1L} for a discussion. Seeing that the map $M\to \Lambda$ 
captures the ``combinatorics'' implicit in the proof of Tietze extension theorem and, arguably,
the notion of contractability, is
slightly less obvious, see the proof of Lemma~\ref{s.2.2.1L}(2).


Everything in this note is very elementary: a reader is likely to improve upon our claims,
and any proofs can be given as exercise to any student familiar with the terminology. 
\subsubsection*{Structure of the paper} As a warm-up the reader may want to skip,
\S\ref{s.2.1}-\ref{s.2.1p} define {\em connected}, {\em quotient}, and {\em compact} in terms of maps of spaces with at most 2 points,
as 
$$\{\emptyset\to\{o\}\}^{rll}\text{, }\{\emptyset\to\{o\}\}^{lrrrl}\text{, and }
\left(\left\{ \{o\} \longrightarrow \left\{ o 
	\rotatebox{-12}{\ensuremath{\to}}\raisebox{-2pt}{\ensuremath{c}}\right\} \right\}^r_{<5}\right) ^\lr$$
In \S\ref{s.2.1} we also define a few other notions 
starting with the simplest possible map, the inclusion of the empty space into a singleton,
and in Appendix~\S\ref{s.3.2} we list a few more. In \S\ref{s.2.1p} we define the class of proper maps of nice spaces. 

\S\ref{s.2.2} and \S\ref{s.2.3} is the main body of the paper. In \S\ref{s.2.2} we discuss
the definition of trivial fibrations, and in \S\ref{s.2.3} we discuss the definitions of fibrations,
trivial fibrations, and Michael selection theory. This enables
us to conjecture a ``finitistic''/computational model structure in \S\ref{s.2.5}.

In Appendix \S\ref{s.3.1} we state \cite[Thm.1.2]{6} of continuous selection theory we use,
and Appendix \S\ref{sBourbaki} we state the theorems of \cite{Bourbaki} we use for compactness.

\section{Observations}\label{s.2}

A number of basic notions in topology can be concisely defined, 
often starting from simplest examples, by  
repeatedly taking the orthogonal with respect to the Quillen lifting property
in the category of topological spaces \cite{1,2,3}. 

Here is a sample: {\em connected, compact, and contractible}; see \cite{2} for a longer list.

\subsection{Connected}\label{s.2.1} Being connected can be defined using the simplest possible map,
the embedding of the empty set into a singleton. 

\begin{lemm}[{$\emptyset\to \{o\}$}]\label{s.2.2.1} In the category of (all) topological spaces,
\addtocounter{enumi}{1}
\begin{enumerate}
\item[r:] $\{\emptyset\to\{o\}\}^r$ is the class of surjections

\item[rl:] $\{\emptyset\to\{o\}\}^{rl}$ is the class of maps $A\to A\sqcup D$ where $D$ is discrete

\item[rllr:] $\{\emptyset\to\{o\}\}^{rllr}$ is the class of maps $A\to A\sqcup D$

\item[rr:] $\{\emptyset\to\{o\}\}^{rr}$ is the class of {\em subsets}, i.e.~the inclusions $i:A\to B$ where $A$ is a subset of $B$, and $i(a)=a$, $a\in A$. 

\item[lrrrl:] $\{\emptyset\to\{o\}\}^{lrrrl}$ is the class of {\em quotients}, i.e. the maps $f:A\to B$ such that 
a subset $U\subset B$ is open in $B$ iff its preimage $f^{-1}(U)\subset A$ is open in $A$.

\item[rll:] A map $f:A\to B$ of ``nice'' spaces belongs to $\{\emptyset\to\{o\}\}^{rll}$ 
iff the induced map $\pi_0(f):\pi_0(A)\to \pi_0(B)$ of connected components is surjective.
In particular, \begin{itemize}

\item A topological space $X$ is connected iff  for each, equiv.~any, map $\{o\}\to X$ 
from a singleton it holds  
$$\{o\}\to X \,\in \, \{\emptyset\to\{o\}\}^{rll}$$ 
\end{itemize}

\end{enumerate}
\end{lemm}

Here in (r) and (rl), $A\sqcup D$ denotes the disconnected union of $A$ and $D$, i.e. both 
subsets $A$ and $D$ are closed and open, and the topology on both $A$ and $D$ is induced.

In (rll), by a space being ``nice'' we mean that it splits into a disconnected union of closed 
and open connected components. 
\begin{proof}
1. By definition 
$$ \{\emptyset\to \{o\}\}^r := \left\{ X\xra{g} Y: \emptyset\to \{o\} \rtt X\xra{g} Y \right\}$$
 is the class of maps which have the right lifting property with respect
to the embedding of the empty subset into a singleton. This lifting property says that any point of $Y$ (the image of $\{o\}$ in $Y$)
has a preimage in $X$ (the image of $\{o\}$ in $X$), i.e. is surjective.
2. By definition  
$$ \{\emptyset\to \{o\}\}^{rr} = \left\{ X\xra{g} Y: f\rtt g \text{ for any } f\in \left\{\emptyset\to \{o\}\right\}^r\right\}$$
is the class of maps which have the right lifting property with respect
to any surjection. If map $g:X\to Y$ represents a subset, i.e.~$X\subset Y$, the topology on $X$ is induced from $Y$, and 
and $f_{|X}=\id_{|X}$, then the image of $B\to Y$ is contained in $X$, and, as the topology on $X$ is induced,  
the lifting is continuous. In the opposite direction, take $B$ to be the image of $g:X\to Y$, 
and $A$ to be the preimage of $g:X\to Y$ with topology induced from $Y$. Then $f \rtt g$ lifts iff $g:X\to Y$ 
represents a subset.
 Rest is similar. 
\end{proof}

\subsection{Compact}\label{s.2.1p}
Perhaps the simplest example of a map which is not closed (and thereby not proper),
is the embedding of a point
	as the open point in the two-point space  with one point open and one point closed.
We denote this map by $\{o\} \longrightarrow \left\{ o 
	\rotatebox{-12}{\ensuremath{\to}}\raisebox{-2pt}{\ensuremath{c}}\right\}$.
\begin{lemm}[{$ \{o\} \longrightarrow \left\{ o 
	\rotatebox{-12}{\ensuremath{\to}}\raisebox{-2pt}{\ensuremath{c}}\right\}$}]\label{s.2.2.1p}
In the category of (all) topological spaces, the class 
$\left(\left\{ \{o\} \longrightarrow \left\{ o 
	\rotatebox{-12}{\ensuremath{\to}}\raisebox{-2pt}{\ensuremath{c}}\right\} 
\right\}^r_{<5}\right) 
^\lr$ is a class of proper maps, and 
\begin{itemize}
\item a map of ``nice'' spaces is proper iff it lies in  
$\left(\left\{ \{o\} \longrightarrow \left\{ o 
	\rotatebox{-12}{\ensuremath{\to}}\raisebox{-2pt}{\ensuremath{c}}
\right\}\right\}^r_{<5} 
\right) ^\lr$
\end{itemize}
In particular, a Hausdorff space $K$ is compact iff 
$$K\to\{o\}\,\in\, \left(\left\{ \{o\} \longrightarrow \left\{ o 
	\rotatebox{-12}{\ensuremath{\to}}\raisebox{-2pt}{\ensuremath{c}}
\right\}\right\}^r_{<5} 
\right) ^\lr
$$
\end{lemm}
Here,  ``nice'' may be taken to mean Hausdorff hereditary normal (separation axioms $T_1$ and $T_5$),
and $\left\{ \{o\} \longrightarrow \left\{ o 
	\rotatebox{-12}{\ensuremath{\to}}\raisebox{-2pt}{\ensuremath{c}}
\right\}\right\}^r_{<5}$ denotes the subclass of $\left\{ \{o\} \longrightarrow \left\{ o 
	\rotatebox{-12}{\ensuremath{\to}}\raisebox{-2pt}{\ensuremath{c}}
\right\}\right\}^r$ consisting of maps of spaces with less than 5 points. 

\begin{proof} See  \S\ref{compacts} or \cite[\S2.2]{mintsGE} for a verbose explanation; here we are brief. 
First check that a map $f$ of finite spaces is closed, equiv.~proper, iff 
$\{o\} \longrightarrow \left\{ o 
	\rotatebox{-12}{\ensuremath{\to}}\raisebox{-2pt}{\ensuremath{c}}\right\} \rtt f$. 
	The definition of being proper via ultrafilters (see Bourbaki \cite[I\S10.2,Th.1(d)]{Bourbaki}, quoted in \S\ref{sBourbaki}) 
	expresses being proper as a lifting property with respect to a class of maps associated with ultrafilters: $f$ is proper iff
	$$ A\to A\sqcup_{\mathcal U} \{\infty\} \,\rtt\, X\xra f Y$$ where the topology on 
	$A\sqcup_{\mathcal U} \{\infty\}$ is such that $\infty$ is closed, $\mathfrak{U}$ is the neighbourhood filter of $\infty$,
	and the topology on $A$ is induced \cite[I\S6.5, Def.5, Example]{Bourbaki}. These maps belong to 
 $\left(\left\{ \{o\} \longrightarrow \left\{ o 
	\rotatebox{-12}{\ensuremath{\to}}\raisebox{-2pt}{\ensuremath{c}}
\right\}\right\}^r_{<5}\right) 
 ^l$, hence any map in $\left(\left\{ \{o\} \longrightarrow \left\{ o 
	\rotatebox{-12}{\ensuremath{\to}}\raisebox{-2pt}{\ensuremath{c}}
\right\}\right\}^r_{<5}\right) 
 ^l$ is proper.
 
Smirnov-Vulikh-Taimanov theorem \cite[3.2.1,p.136]{Engelking} gives sufficient conditions 
to extend a map to a compact Hausdorff space, and can be generalised to give the required 
lifting property. It says that a map to a compact Hausdorff space can be extended to the whole space $X$
from a dense subset $A$ satisfying (in fact the necessary) condition 
for every pair $B_1,B_2$ of disjoint closed subsets of \ensuremath{A} the inverse images $f^{-1}(B_1)$
and $f^{-1}(B_2)$ have disjoint closures in the space $X$. 
A verification shows that the following four maps are closed and their left orthogonals define these sufficient conditions 
on $A\to X$:\footnote{Our notation represents finite topological space as preorders or finite categories with each diagram commuting, \
and is hopefully self-explanatory; see \cite{3} for details. 
In short, 
an arrow $o\to c$ indicates that $c\in \operatorname{cl}\,o$, and each point goes to ``itself''; 
the list in $\{..\}$ after the arrow indicates new relations/morphisms added, thus in $\{o\to c\}\lra \{o=c\}$ 
the equality indicates that the two points are glued together or that we added an identity morphism between $o$ and $c$.
The notation in the 3rd line informal (red indicates new/added elements), and in the 4th line reminds of a computer syntax.}
$$\hskip-42pt\begin{array}{ccccc}
	\{a\raisebox{6pt}{}\rotatebox{12}{\ensuremath{\leftarrow}}\raisebox{2pt}{\ensuremath{u}\raisebox{4pt}
	{}\rotatebox{-13}{\ensuremath{\rightarrow}}} b\}\lra\{a=u=b\} 	
	 &
\{a\llrra b\}\lra \{a=b\} &  \left\{{\ensuremath{o}}\rotatebox{-12}{\ensuremath{\to}}  \raisebox{-2pt}{c}\right\}\lra \{o=c\} & 
\{c\}\lra \left\{{\ensuremath{o}}\rotatebox{-12}{\ensuremath{\to}}  \raisebox{-2pt}{c}\right\} &
\\ \text{(disjoint closures)}        & \text{(injective)}            & \text{(pullback topology)}  & \text{(dense image)}&
\\ \hskip3pt	\{a\raisebox{6pt}{\color{red}\bf\,=}\!\!\!\!\!\!\rotatebox{12}{\ensuremath{\leftarrow}}\raisebox{2pt}{\ensuremath{u}\raisebox{4pt}{\color{red}\bf\,=}\!\!\!\!\!\!\rotatebox{-13}{\ensuremath{\rightarrow}}} b\} & 
\{a \leftrightarrow \!\!\!\!\!\!\raisebox{6pt}{{\color{red}\bf\!\!=\,\,}} b\}& 
\{\raisebox{0pt}{\ensuremath{o}}\raisebox{6pt}{\color{red}\bf\,=}\!\!\!\!\!\!\rotatebox{-12}{\ensuremath{\to}}  \raisebox{-2pt}{c}\}&
		\{\raisebox{0pt}{\ensuremath{o}}{\color{red}\bf\rotatebox{-12}{\ensuremath{\to}}  \raisebox{-2pt}{c} }\}

\\	\verb|{a<-u->b}-->{a=u=v}|& \verb|{a<->b}-->{a=b}|& \verb|{o->c}-->{o=c}|& \verb|{c}-->{o->c}|&
\end{array}$$
Hence, the Smirnov-Vulikh-Taimanov theorem \cite[3.2.1,p.136]{Engelking} implies that a Hausdorff space $K$ is compact iff $K\to\{o\}$ is in\newline
$\left\{
\{a\raisebox{6pt}{}\rotatebox{12}{\ensuremath{\leftarrow}}\raisebox{2pt}{\ensuremath{u}\raisebox{4pt}
	{}\rotatebox{-13}{\ensuremath{\rightarrow}}} b\}\lra\{a=u=b\} 	
	 ,
\{a\llrra b\}\lra \{a=b\} ,
  \left\{{\ensuremath{o}}\rotatebox{-12}{\ensuremath{\to}}  \raisebox{-2pt}{c}\right\}\lra \{o=c\} ,
\{c\}\lra \left\{{\ensuremath{o}}\rotatebox{-12}{\ensuremath{\to}}  \raisebox{-2pt}{c}\right\} 
\right\}^{lr},$\newline	 and the latter is a subclass of  $\left(\left\{ \{o\} \longrightarrow \left\{ o 
	\rotatebox{-12}{\ensuremath{\to}}\raisebox{-2pt}{\ensuremath{c}}
\right\}\right\}^r_{<5}\right) 
 ^{lr}$.
\end{proof}
Is it useful to say that these four maps of preorders reveal combinatorics implicit in the notion of compactness ? 

Note that for this statement it is important that the category of topological spaces contains spaces associated with ultrafilters
that would usually be considered to belong to wild phenomena such as curves cheerfully filling cubes,
which are irrelevant from the point of view of the topological intuition of shapes,
cf.~\cite[\S5, pp.28/29]{4}.

\subsection{Contractible}\label{s.2.2}
To define {\em contractible} (among  ``nice'' spaces), 
it is enough to consider a morphism $M\to\Lambda$ from a space $M$ with 5 points
(two open and three closed), into a space $\Lambda$ with 3 points (one open and two closed), see~Fig.~1 
\begin{lemm}[{$M\to \Lambda$}]\label{s.2.2.1L} In the category of (all) topological spaces,	
$\left\{ M\to \Lambda\right\}^{lr}$ is a class of trivial Serre fibrations, and 
\begin{enumerate}
\item A ``nice'' space $Y$ is contractible iff $$Y\to \{o\} \in \left\{ M\to \Lambda \right\}^{lr}$$
\item $X$ is normal (not necessarily Hausdorff) iff $\emptyset\to X  \, \in \, \left\{ M\to \Lambda\right\}^{l}$, 
i.e.$$\emptyset\to X \rtt M\to \Lambda$$
\item For a map $A\to X$ 
	from a Hausdorff space $A$ to a ``nice'' (meaning Hausdorff hereditary normal) space $X$, 
it represents a closed subset $A\subset X$ iff	
$A\hookrightarrow X \, \in \, \left\{ M\to \Lambda\right\}^{l}$, i.e. 
$$A  \hookrightarrow X \rtt M\to \Lambda$$
\end{enumerate}
\end{lemm}\noindent
In (1), ``nice'' may be taken to mean ``being a finite CW complex''.\footnote{As pointed out by Tyrone at \href{https://mathoverflow.net/questions/409266/are-trivial-fibrations-of-finite-cw-complexes-soft-for-normal-maps}{mathoverflow.net}, 
``nice'' may not taken to mean
being a CW complex: Let $C\mathbb{N}$ be the cone over a countably infinite discrete complex (this is a contractible 1-dimensional polyhedron). 
van Douwen and Pol [van Douwen, Eric K.; Pol, Roman. Countable spaces without extension properties. 
Bull. Acad. Polon. Sci. Sér. Sci. Math. Astronom. Phys. 25 (1977), no. 10, 987--991.] 
have constructed a countable regular $T_2$ space $X$ (which is thus perfectly normal) 
and a function $A\to C\mathbb{N}$, 
defined on a certain closed $A\subset X$, 
which does not extend over any neighbourhood in $X$. 
In particular, the map of countable complexes $C\mathbb{N}\to\{o\}$ 
is both a Hurewicz fibration and a homotopy equivalence, but is not soft wrt all perfectly normal pairs.} 
What we need is
that $Y$ is a retract of some Euclidean space $\mathbb{R}^n$ iff $Y$ is weakly contractible. 

Of course, this Lemma tempts a conjecture
\begin{conj}\label{trivialMcVlr} A map of ``nice'' spaces is a trivial fibration iff 
it belongs to $\left\{ M\to \Lambda\right\}^{lr}$.
\end{conj}

\begin{proof} Recall that 
$$ \left\{ M\to \Lambda \right\}^{lr}=
\left\{\,Y\xra p B\,:\,  A\xra i X \rtt Y\xra p B \text{ whenever } A\xra i X \rtt M\to\Lambda \right\} $$
Thus, to see that $\left\{ M\to \Lambda\right\}^{lr}$ is a class of trivial Serre fibrations 
it is enough to verify that $\mathbb{S}^n\to \mathbb{D}^{n+1}\in  \left\{ M\to \Lambda\right\}^{l}$,
where $\mathbb{S}^n\to \mathbb{D}^{n+1}$ denotes the standard embedding of an $n$-sphere into the $n+1$-ball as the boundary.
We skip this, and only remark that to verify that 
$\mathbb{S}^n\to \mathbb{D}^{n+1} \rtt M\to \Lambda$ we need to use that $\mathbb{D}^{n+1}$ is hereditary normal.\footnote{
Namely, use the following characterisation: a space is hereditary normal iff whenever each of two disjoint subsets can be 
	separated from the other by an open neighbourhood, they have disjoint open neighbourhoods. \cite{3} 
	represents this as a lifting property.} 
(2). To give a map $X\lra \Lambda$ is to give two disjoint closed subsets of $X$; 
to give a lifting to $M$ is to find their disjoint neighbourhoods.
(1). It is enough to show that for $Y=[0,1]$: 
indeed, ${}^r$-orthogonals are closed under products and retracts, and any contractible finite CW complex
is a retract of some $[0,1]^n$, $n>0$ \cite{10}.  The proof for $Y=[0,1]$ we give is the standard proof of the Tietze extension theorem 
retold in a diagram chasing notation. 

 Represent the interval $[0,1]$ as a union
 $$[0,1]= \{0\}\cup (0,t_1) \cup \{t_1\} \cup (t_1,t_2) \cup 
...\cup (t_{n-1},1)\cup \{1\}$$ 

Contract the open intervals to (open) points, and denote the resulting map by 
$[0,1]\to \Lambda_n$ where $\Lambda_n=
\left\{\underset{0}{}{\swarrow} \overset{\raisebox{1pt}{$t_{0,1}^-$}}{}{\searrow} \underset{t_1}{}{\swarrow}\overset{\raisebox{1pt}{$t_{1,2}^-$}}{}{\searrow}\underset{t_2}{}{\swarrow}....{\searrow}
\underset{t_{n-2}}{}{\swarrow} \overset{\raisebox{1pt}{$t_{n-2,n-1}^-$}}{}{\searrow} \underset{t_{n-1}}{}{\swarrow}\overset{\raisebox{1pt}{$t_{n-1,n}^-$}}{}{\searrow}\underset{1}{}
\right\}$.  Subdividing the open intervals gives maps $\Lambda_{2n}\to \Lambda_n$.
The map $\Lambda_1=M\to \Lambda=\Lambda_0$ corresponds to subdividing a single open interval into two. 
Use that $^r$-orthogonals are closed under pullbacks to see that $\Lambda_{2n}\to \Lambda_n \in \left\{ M\to \Lambda \right\}^{lr} $,
and that $^r$-orthogonals are closed under inverse limits to see that $\Lambda_\omega\to \Lambda \in \left\{ M\to \Lambda \right\}^{lr} $
where $\Lambda_\omega:=\lim\limits_{\Lambda_{2n}\to\Lambda_n} \Lambda_{2^n}$
and that $^r$-orthogonals are closed under composition to see that 
$\Lambda_\omega\to \Lambda \in \left\{ M\to \Lambda \right\}^{lr} $,
and that $^r$-orthogonals are closed under composition to see that 
$\Lambda_\omega\to \{o\} \in \left\{ M\to \Lambda\right\}^{lr}$
as $ \Lambda\to\{o\}$ is a retract of $\Lambda_4\to \Lambda_2$. 
 Finally, the maps $[0,1]\to \Lambda_n$ induce an embedding $[0,1]\to \Lambda_\omega$ of $[0,1]$
 into $\Lambda_\omega$ as a retract, hence, an orthogonals are closed under retract, 
 we get the required result. 
(3). Pick a map sending $X$ to the open point of $\Lambda$, and the separating neighbourhoods
of two distinct points of $A$ to the two open points of $M$. A lifting would provide separating neighbourhoods
of their images. Therefore, the map $A\to X$ is injective. To see that it is closed, 
pick a map sending the whole of $A$ to the closed point in the ``middle'' of $M$,
and an arbitrary point $x$ of $X-A$ into a closed point of $\Lambda$. 
 A lifting would provide neighbourhood of $x$ disjoint from $A$. To see that the topology on $A$ is induced,
Pick a map $X\to \Lambda$ sending $X$ to the open point of $\Lambda$, and 
a map $A\to M$ sending an arbitrary open subset  $U$ of $A$ 
into an open point of $\Lambda$.  A lifting would provide 
an open subset of $X$ whose intersection with $A$ is $U$.
 \end{proof}

\subsection{The naive defining lifting property of a fibration}\label{s.2.3}\label{s2.3}
If all spaces were ``nice'', we could perhaps define fibrations and trivial fibrations as follows:
\begin{itemize}

\item[(wf)] a map $p:Y\to B$ is a trivial fibration 
iff the lifting property  
$A\to X \rtt  Y\xra p B$
holds whenever $A\subset X$ is a closed subset of a space $X$. 

\item[(f)] 
a map $p:Y\to B$ is a fibration  
iff, whenever $A\subset X$ closed, for any lifting problem  
$A\to X \rtt  Y\xra p B$, 
there exists a diagonal lifting  defined on some open neighborhood of $A$.\footnote{
Formally in notation, for any commutative square
$$\xymatrix{ A \ar[r]^{f} \ar@{->}[d]_{i} & {Y} \ar[d]^{p} \\ {X}  \ar[r]^{\phi} & {B} }$$
there is an open $A\subset U\subset X$ and a map $\tilde f_U:U\to Y$ such that the diagram 
$$\xymatrix{ {} & A \ar[r]^{f} \ar@{..>}[dl] \ar@{->}[d]_{} & {Y} \ar[d]^{p} \\ 
{U}\ar@{..>}[r] \ar@{..>}[urr]|{\tilde f_U} & X  \ar[r]|{\phi_{|U}} & {B} }$$
commutes.}
\end{itemize}

In (f), we get the definition of trivial Serre fibration if we restrict $A\subset X$
to be cellular inclusions of finite CW complexes, or indeed just 
the inclusions $\mathbb{S}^n\to \mathbb{B}^{n+1}$ of an $n$-sphere as the boundary of $n+1$-ball, $n\geq 0$.


Michael selection theory (see \S\ref{selection}) says that we do get the standard notions 
of a trivial fibration, and of a fibration, if we take $X$ to vary 
among paracompact spaces of finite  Lebesgue dimension; then  
it is sufficient for $p:Y\to B$ to be  a map of complete metric spaces with uniformly contractible fibres,
i.e.~a map of topological spaces admitting complete metrics 
such that there are $\delta,\varepsilon>0$ such that 
in any fibre any ball of radius $\delta$ is contractible within a ball of radius $\varepsilon$ (in the fibre).
These assumptions come from Michael continuous selection theory \cite[Thm.1.2]{6}, see \S\ref{selection}.

We rewrite (wf) and (f) in the diagram chasing manner
using Lemma~\ref{McVlrfwf} and the notion of non-Hausdorff mapping cone/cylinder,\footnote{
Intuitively, 
this is the usual (Hausdorff) mapping cone 
 ${Y\times [0,1]}{/ \{(y,1)=p(y)\}}$
 where we replaced $[0,1]$ by the two-point Sierpinski-Kolmogorov space  
$\left\{o\rotatebox{-12}{\ensuremath{\to}}\raisebox{-2pt}{\ensuremath{c}}\right\}$.
Formally, 
the {\em non-Hausdorff mapping cone/cylinder} of a map $p:Y\to B$, denoted by $Y_{{}_p\!\!\searrow}\underset{B}{}$, is 
$Y\times \left\{o\rotatebox{-12}{\ensuremath{\to}}\raisebox{-2pt}{\ensuremath{c}}\right\}/\raisebox{-2pt}{\ensuremath{(-,c)=p(-)}}$,
i.e.~the disjoint union $Y \sqcup B$ equipped with the following topology: 
an open subset is either an open subset of $X$,
or the union of an open subset of $B$ and its preimage. 
%
} 

\begin{lemm}\label{McVlrfwf} 
In a full subcategory of ``nice'' topological spaces, 
	\begin{itemize}
		\item[(wf)$'$] a ``very nice'' map is a trivial fibration iff it belongs to 
	$ \left\{ M\to \Lambda \right\}^{lr}$

	\item[(f)$'$] a ``very nice'' map is a fibration iff  
the map from its non-Hausdorff mapping cone to the base belongs to
	$ \left\{ M\to \Lambda \right\}^{lr}$
$$Y_{{}_p\!\!\searrow}\underset{B}{}\to B \,\,\in\,\, \left\{ M\to \Lambda \right\}^{lr}$$
	\end{itemize}	
\end{lemm}
Here, being ``nice'' means being (possibly non-Hausdorff) paracompact of finite Lebesgue dimension, 
and ``very nice'' means say a map of finite CW complexes or being smooth in a suitable sense (we need something to ensure that 
a fibration is necessarily a map of complete metrisable spaces with uniformly locally contractible fibres),
and $Y_{{}_p\!\!\searrow}\underset{B}{}$ denotes the non-Hausdorff mapping cone of $Y\xra p B$.
\begin{proof} 
Recall that 
$$ \left\{ M\to \Lambda \right\}^{lr}=
\left\{\,Y\xra p B\,:\,  A\xra i X \rtt Y\xra p B \text{ whenever } A\xra i X \rtt M\to\Lambda \right\} $$
A map to a Hausdorff spaces necessarily glues together points which cannot be separated by neighbourhoods 
(for their images can if distinct), hence we may assume that both $A$ and $X$ are Hausdorff
and by Lemma~\ref{McVlrfwf}(3) that $A\xra i X$ is the inclusion of a closed subset. 
Hence, (f)$'$ states precisely (f) above, i.e.~the conclusion of 
Michael selection theorem Theorem~\ref{MTh1.2}  for trivial fibrations. 

Similarly, (wf)$'$ is (wf)  using 
the diagram chasing property of the non-Hausdorff mapping cone: 
\begin{itemize}\item is to give a map $X\to \raisebox{3pt}{$Y_{{}_p\!\!\searrow}\underset{B}{}$}$
is the same as to give a commutative square 
$$\xymatrix{ U \ar[r]^{f} \ar@{->}[d]_{i} & {Y} \ar[d]^{p} \\ {X}  \ar[r]^{\phi} & {B} }$$
for some open subset $U$ of $X$.
\end{itemize} 
Indeed, this means that the lifting property $A\xra i X \rtt\,\, \raisebox{3pt}{$Y_{{}_p\!\!\searrow}\underset{B}{}$}\xra p B$ 
of item (f)$'$ holds 
iff    
for any open subset $U$ of $A$ and a commutative square
$$\xymatrix{ U \ar[r]^{f} \ar@{->}[d]_{i} & {Y} \ar[d]^{p} \\ {X}  \ar[r]^{\phi} & {B} }$$
there is an open $U\subset V\subset X$ and a map $\tilde f_V:V\to Y$ such that the diagram 
$$\xymatrix{ U \ar[r]^{f} \ar@{->}[d]_{} & {Y} \ar[d]^{p} \\ {V}  \ar[r]|{\phi_{|V}} \ar@{..>}[ur]|{\tilde f_V}& {B} }$$
commutes. 
 This is almost the conclusion of 
Michael selection theorem Theorem~\ref{MTh1.2}  for fibrations as stated.

Finally, by Lemma~\ref{McVlrfwf}(3), $ \left\{ M\to \Lambda \right\}^{l}$ contains
the inclusion $\mathbb{S}^n\to \mathbb{B}^{n+1}$ of an $n$-sphere as the boundary of $n+1$-ball, $n\geq 0$,
and thereby $ \left\{ M\to \Lambda \right\}^{lr}$ is a subclass of trivial Serre fibrations.
The reader will find it an exercise (check this!) to see that (f)$'$ implies that $Y\xra p B$ is a fibration
under suitable assumptions. 
\end{proof}

\subsection{A naive ``combinatorial'' model structure}\label{s.2.5}
 Considerations above suggest the following conjecture. The idea is to
 use $M\to \Lambda$ to make precise niceness assumptions in 
 the naive lifting property of fibrations. 
\begin{conj}[{$M\to \Lambda$}]\label{s.2.4.2} A closed model structure on the category of topological spaces
is defined as follows:
\begin{itemize}
\item $\left\{ M\to \Lambda\right\}^{l}$ is the class of cofibrations.
\item $\left\{ M\to \Lambda \right\}^{lr}$ is the class of trivial fibrations.
\item $\left\{ Y\xra p B \,:\, 
|Y|<\infty, |B|<\infty, \text{ and }  \raisebox{3pt}{$Y_{{}_p\!\!\searrow}\underset{B}{}$} \in \left\{ M\to \Lambda \right\}^{lr}\right\}^{l}$
is the class of trivial cofibrations. 
\item $\left\{ Y\xra p B \,:\, 
|Y|<\infty, |B|<\infty, \text{ and }  \raisebox{3pt}{$Y_{{}_p\!\!\searrow}\underset{B}{}$} \in \left\{ M\to \Lambda \right\}^{lr}\right\}^{lr}$
is the class of fibrations. 
\item a weak equivalence is the composition of a trivial cofibration with a trivial fibration

\end{itemize}
\end{conj}

The language of this conjecture is purely combinatorial. Can we define a model category of ``formal'' topological spaces
(``formal'' as in formal power series), i.e.~a model category whose objects and arrows belong to 
a calculus of diagram chasing computations, so to say ? 
A naive hope is that the size of spaces appearing in the Quillen orthogonals (negations) representing basic notions of topology \cite{2,3}
is small enough ($<7$) to make feasible the exponential growth in the computer processing of such a calculus. 

The following would represent a rule in such a diagram chasing calculus of formal topological spaces.
\begin{conj}[M2] For each finite set $P$ of maps of finite spaces, and each string consisting of letters $l$ and $r$,
each map in the category of topological spaces decomposes as a map in $(P)^{sl}$ followed by a map in $(P)^{slr}$, 
and as a map  in $(P)^{srl}$ followed by a map in $(P)^{sr}$:
$$
\xymatrix{ {}  & {\cdot} \ar@{..>}[rd]|{(P)^{slr}} & \\ {\cdot} \ar[rr]|\forall \ar@{..>}[ur]|{(P)^{sl}} & & \cdot }
\xymatrix{& & & }
\xymatrix{ {}  & {\cdot} \ar@{..>}[rd]|{(P)^{sr}} & \\ {\cdot} \ar[rr]|\forall \ar@{..>}[ur]|{(P)^{srl}} & & \cdot }$$

\end{conj}

Of course, the real temptation is to develop a computer algebra system doing topology using a syntax
extending the concise syntax for topology we discuss, and to use it in teaching. 

\section{Appendix.}

\subsection{Appendix. Michael continuous selections\label{selection}}\label{s.3.1}
We sketch the statement of the Michael continuous selections theorem \cite[Thm.1.2]{6} we use, see also \cite{5,7}.

Let $(F_x)_{x\in X}$ be a family of non-empty subsets of a topological space $Y$. 
Michael selection theory thinks of such a family as a multivalued function $\phi:X\to 2^Y$
and refers to the family as a {\em carrier}.
Michael selection theory gives sufficient conditions for existence of 
a continuous choice function $f(x)\in F_x,x\in X$. 
These conditions are satisfied when the family $(F_x)_{x \in X}$ 
is the family of fibres of a fibration of "nice" spaces.   
\cite{7} considers families of convex subsets of a Banach space 
but we do not discuss it here.

The family $(F_x)_{x\in X}$ is {\em lower semi-continuous} iff, whenever $U\subset Y$ is open in $Y$, 
the subset $\{x \in X : F_x \cap U \neq \emptyset\}$ is open in $X$. This subset 
can be thought of as the preimage of $U$ under the  multivalued function $(F_x)_{x\in X}$.

The family $(F_x)_{x\in X}$  is
 {\em uniformly locally $n$-contractible} iff, for every $x\in X$ and every $y \in F_x$,
 and every neighbourhood $U\in y$  of $y\in Y$, there exists a neighbourhood
$V\ni y$ of $y \in Y$ such that, for every $F_{x'}, x'\in X$, 
every
continuous image of an $m$-sphere ($m \leq n$) in $F_{x'} \cap V$ 
is contractible in $F_{x'} \cap U$. By convention, as there are no $m$-spheres for $m<0$,
we assume each family is 
uniformly locally $-1$-contractible. 
As a diagram  
$\forall x\in X\, \forall y \in F_x \,\forall U_y\ni y \, \exists V_y\ni y,\,V_y\subset U_y
\forall x' \in X$
$\xymatrix{ \mathbb{S}^n \ar[d]\ar[r]^\forall & F_{x'}\cap V_y\ar[d] \\
\mathbb{B}^{n+1} \ar@{-->}[r]^\exists & F_{x'}\cap U_y \\ }$

Let $\dim X$ denote the {\em Lebesgue (covering) dimension}; i.e., for normal space $X$, 
$\dim X \leq n$ iff  $A\to X \rtt \mathbb{S}^n\to \{o\}$ 
for every closed subset $A\subset X$. A space is paracompact iff every open covering has a locally finite subcovering.\footnote{The usual definition is in terms of open coverings. 
We combine \cite[\S9]{7} and \cite{8}:

 "A {\em open covering} of a topological space $X$ is, in [\cite{9}], a collection of open subsets
of X whose union is X. Its elements need
not be open unless that is specifically assumed. A {\em refinement} of a covering $\mathcal U$ is a covering $\mathcal V$ such that
every $V \in \mathcal V$ is a subset of some $U \in \mathcal U$. A covering $\mathcal U$ is {\em point-finite}
if every $x \in X$
is an element of only finitely many $U \in \mathcal U$, it is {\em locally finite} if every $x \in X$ has a
neighbourhood intersecting only finitely many $U \in \mathcal U$. 

Call a collection $\mathcal U$ of subsets of a topological space {\em closure-preserving} 
if, for every subcollection $\mathcal V\subset \mathcal U$ the union of closures is the
closure of the union (i.e.~$\cup \{ \bar U: U\in \mathcal U\}=[\cup \{  U: U\in \mathcal U\}]^-$). Any locally
finite collection is certainly closure-preserving, but the converse is
generally false even for discrete spaces. 

A Hausdorff space $X$ is called {\em paracompact} iff every
open covering of $X$ has a locally finite open refinement. 
For a regular space it is equivalent to require only that  
every open covering has a closure-preserving refinement \cite[Thm.1]{9}.

By dimension, or $\dim$, we mean the {\em Lebesgue (covering) dimension}; i.e., 
$\dim X \leq n$ 
iff every finite open covering $\mathcal U$ of $X$ has a finite, open refinement $\mathcal V$ of order $\leq n$ 
(i.e.~every $x \in X$ is in at most $n + 1$ elements of $\mathcal V$). If $A \subset X$ is closed, then we say that
$\dim_X(X - A) \leq n$ if $\dim(C) \leq n$ for every $C \subset X - A$ which is closed in $X$; for metric $X$,
this is equivalent to $\dim(X - A) \leq n$. 
}\footnote{These notions can probably be expressed as lifting properties as follows.
To give a finite open, resp.~closed, covering $\mathcal U$  is to give a map $X\lra \{\mathcal V\,:\,\emptyset\neq  \mathcal V \subset \mathcal U \} $
where the topology is defined by the order $\mathcal V_1\to \mathcal V_2$ iff $\mathcal V_1\supset \mathcal V_2$, resp.~$\mathcal V_1\subset \mathcal V_2$.
To give a finite open covering $\mathcal U$ of order $\leq n$ is to give a map $X\lra \{\mathcal V\,:\,\emptyset\neq  \mathcal V \subset \mathcal U, |\mathcal V|\leq n+1 \} $.
A finite open covering $\mathcal U$ of $X$ has a finite, open refinement $\mathcal V$ of order $\leq n$ 
iff $\emptyset\to X \,\rtt\, \{(\mathcal W,\mathcal V): \emptyset\neq \mathcal W\subset \mathcal V \subset \mathcal U, |\mathcal W|\leq n+1 \}\to \{\mathcal V\,:\,\emptyset\neq  \mathcal V \subset \mathcal U \} $
 where the topology is generated by the orders
$(\mathcal W_1,\mathcal V_2)\to (\mathcal W_2,\mathcal V_2)$ iff $\mathcal W_1\supset \mathcal W_2$ and $\mathcal V_1\supset \mathcal V_2$, and  
$\mathcal V_1\to \mathcal V_2$ iff $\mathcal V_1\supset \mathcal V_2$.

To give a point-finite closure-preserving closed covering $\mathcal U$ of $X$ is to give a map 
$X\lra \{\mathcal V\,:\,\emptyset\neq  \mathcal V \subset \mathcal U, |\mathcal V|< \omega \} $
where the topology is defined by the order
$\mathcal V_1\to \mathcal V_2$ iff $\mathcal V_1\subset \mathcal V_2$(sic!).
An open covering $\mathcal U$ has a point-finite closure-preserving refinement $\mathcal V$ iff
$\emptyset\to X \,\rtt\, \{(\mathcal W,\mathcal V): \emptyset\neq \mathcal W\subset \mathcal V \subset \mathcal U, |\mathcal W|<\omega \}\to \{\mathcal V\,:\,\emptyset\neq  \mathcal V \subset \mathcal U \} $ where the topology on the domain is defined by order
$(\mathcal W_1,\mathcal V_2)\to (\mathcal W_2,\mathcal V_2)$ iff $\mathcal W_1\subset \mathcal W_2$(sic!) and $\mathcal V_1\supset \mathcal V_2$,
on the target by the open subsets 
$\{\mathcal V\subset \mathcal U: U\in \mathcal V\}$, for $U\in\mathcal U$.
}

\begin{theo}[{\cite[Thm.1.2]{6}}]\label{MTh1.2} 
Let $X$  be a paracompact Hausdorff space,
$A \subset X$ closed with $\dim_X(X - A) \leq n + 1$, 
and let $(F_x)_{x\in X}$ be a uniformly locally $n$-contractible family 
of non-empty closed subsets of a complete metric space $Y$. 

Then every continuous choice function on $A$ extends to a 
continuous choice function on an open neighborhood of $A$.
Moreover, if every $F_x,x\in X$ is $n$-contractible, then
every continuous choice function on $A$ extends to a 
continuous choice function on the whole of $X$.
\end{theo}

We repeat the conclusion in notation: 
for every continuous choice function $f:A\to Y$ such that $f(x)\in F_x$ whenever $x\in A$,
there is an open neighbourhood $U\subset A$ of $A$ and a continuous choice function 
$\tilde f:U\to Y$ such that $\tilde f(u) \in F_u$ whenever $u\in U$, 
and $f(a)=\tilde f(a)$ whenever $a\in A$.

\subsection{Extending maps to compact spaces}\label{sBourbaki}\label{compacts}

We explain  in more detail the proof in \S\ref{s.2.2} of the characterisation of compactness.  
The reader may find a verbose exposition focusing on logical ideas in \cite[\S2.2]{mintsGE}. 

\subsubsection{Compactness via ultrafilters by Bourbaki}

%
Item $\verb|d)|$ of the following characterisation of proper maps by Bourbaki \cite{Bourbaki} states almost a lifting property.
Arguably, this suggests that the ideas/technique of category theory were present in \cite{Bourbaki}, 
although not the notation or language of category theory. 
\newline\noindent \includegraphics[width=\linewidth]{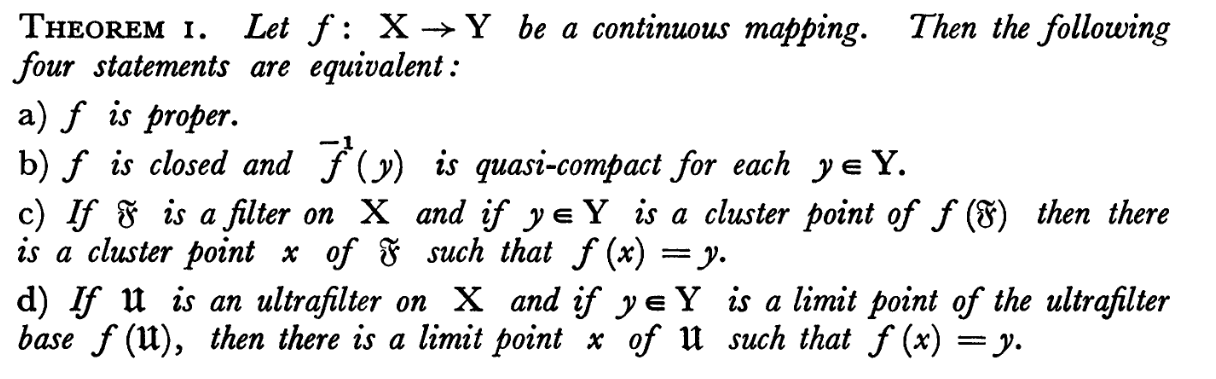}
Item $\verb|d)|$ expresses the following lifting property (almost): 
$|X| \lra |X|\sqcup_{\mathfrak U}\{\oo\} \,\rtt\, X\xra f Y $
where $|X|$ denotes the set of points of $X$ equipped with discrete topology,
and the topology on $|X|\sqcup_{\mathfrak U}\{\oo\}$ is such that 
$\mathfrak U$ is the neighbourhood filter of $\oo$, and the induced topology on 
subset $|X|$ is discrete \cite[I\S6.5, Def.5, Example]{Bourbaki}. 
\subsubsection{Extending maps to compact Hausdorff spaces}
The theorem of Vulikh-Smirnov-Taimanov \cite[3.2.1,p.136]{Engelking} is stated 
in the language of lifting properties almost explicitly (``compact'' below stands for ``compact Hausdorff''):\vskip3pt
\noindent \includegraphics[width=\linewidth]{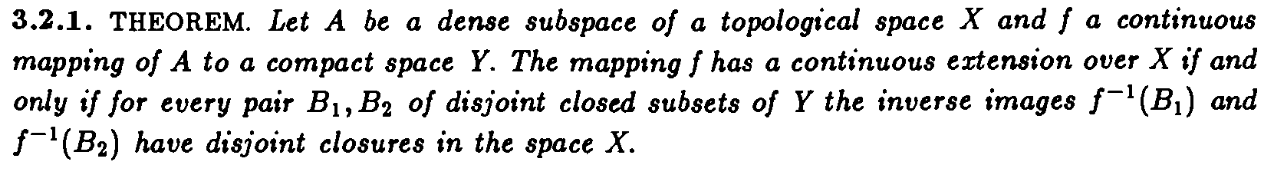}

Let us transcribe this to the language/notation of finite topological spaces and lifting properties. 
We are given {\sf a dense subspace $A\xra i X$ of a topological space $X$}  and 
{\sf a continuous mapping $A\xra f Y$ of $A$ to a [Hausdorff] compact space $Y$}. 
{\sf The mapping \ensuremath{f} has a continuous extension over \ensuremath{X}}
means that the arrow $A\xra f Y$ factors via $A\xra i X$ (cf.~Figure 2f). 
A {\sf   pair $\boldsymbol{B_1}$, $\boldsymbol{B_2}$  of disjoint closed subsets of \ensuremath{Y}} is an arrow 
$Y\lra  \{\boldsymbol{B_1}\leftarrow O \rightarrow \boldsymbol{B_2}\}$
where $\{\boldsymbol{B_1}\leftarrow O \rightarrow \boldsymbol{B_2}\}$ is the space with 
one open point denoted by $O$ and two closed points
denoted by $\boldsymbol{B_1}$ and $\boldsymbol{B_2}$.
To say {\sf  the inverse images $\boldsymbol{f^{-1}(B_1)}$
and $\boldsymbol{f^{-1}(B_2)}$ have disjoint closures in the space $X$} is to say that 
the composition 
$A\xra f Y \lra \{\boldsymbol{B_1}\leftarrow O \rightarrow \boldsymbol{B_2}\}$ factors as $A \xra i X
\lra  \{\boldsymbol{B_1}\leftarrow O \rightarrow \boldsymbol{B_2}\}$ (cf.~Figure 2g).

Now we need to define the class of  dense subspaces. 
A dense subspace is an injective map with dense image 
such that the topology on the domain is induced from the target.
This suggests we try to define this class by
taking left Quillen negations (orthogonals) of the simplest archetypal examples of 
a map whose image is not dense ($\{U\} \lra  \{ U \rightarrow U' \}$), a non-injective map ($\{x\llrra y\} \lra  \{x=y\}$), 
and a map such that the topology on the domain is not induced from the target ($\{o \rightarrow c \} \lra  \{o=c\}$). 


Doing so leads to the following  reformulation. 
\begin{theo}
Let $Y$ be Hausdorff compact and let $A\xra i X$ satisfy (cf.~Figure 2(ijk))
\begin{enumerate}
\item[(i)]  (dense) $A\xra i X \rtt  \{U\} \lra  { \{ U \rightarrow U' \} } $
\item[(ii)] (injective)  $A\xra i X \rtt { \{x\llrra y\} } \lra  \{x=y\}$
\item[(iii)] (induced topology) $A\xra i X \rtt { \{o \rightarrow c \} } \lra  \{o=c\} $
\end{enumerate}
Then the properties of $A\xra f Y$ defined by Figure 2(f) and Figure 2(g) are equivalent.
\end{theo}

This implies that, for Hausdorff compact $Y$, 
items   3.2.1(i-iii) and $ A\xra i X \rtt \{\boldsymbol{B_1}\leftarrow O \ra \boldsymbol{B_2}\}\lra \{\boldsymbol{B_1} = O= \boldsymbol{B_2}\}
$ imply that $ A\xra i X \rtt Y\lra \{\bullet\}$.

Further, note that if $X=A\sqcup\{\oo\}$ is obtained from $A$ by adjoining a single closed non-open point, then 
$$ A\xra i X \rtt \{\boldsymbol{B_1}\leftarrow O \ra \boldsymbol{B_2}\}\lra \{\boldsymbol{B_1} = O= \boldsymbol{B_2}\}$$
iff there exists an ultrafilter $\mathfrak U$ such that $A \xra i X$ is of form $A \lra A\sqcup_{\mathfrak U}\{\oo\}$.

This implies that maps of form  $A \lra A\sqcup_{\mathfrak U}\{\oo\}$ are in $P^l$ and, finally, 
that a Hausdorff space $K$ is quasi-compact iff $K\lra\{\bullet\}$ is in $P^{lr}$ where
$P$ consists of
\begin{center}$ 
{ \{\boldsymbol{B_1}\leftarrow O\rightarrow \boldsymbol{B_2}\} } \lra  {\{\bullet\} } \ \ \ \  \ \ \ \ \
 \{U\} \lra  { \{ U \ra U' \} }   $ \\  
$
{ \{x\llrra y\} } \lra  \{x=y\} \ \ \ \ \ \ \ \ \  \ \ \ \ \ \ 
{ \{o \ra c \} } \lra  \{o=c\}   
$
\end{center}
\subsubsection{A logical point of view: the simplest counterexample negated three times.} 
We took a (the?) simplest possible non-proper map, took Quillen negation thrice (although once passing 
to the subclass of finite spaces), and got (almost?) the definition of a proper map. 

Let us explicitly state the conjecture. 

\begin{enonce*}{Conjecture}[{$(\{ \{o\}\lra \{o\rightarrow c\}\}^{r}_{<5})^{lr}$}] 
In the category of topological spaces, the following Quillen orthogonal (negation) 
defines the class of proper maps:
$$(\{ \{o\}\lra \{o\rightarrow c\}\}^{r}_{<5})^{lr}$$ 
\end{enonce*}

\subsection{Appendix. Properties of the empty subspace of a singleton}\label{s.3.2}
We give a list of properties of maps one can define starting with the 
simplest possible map $\emptyset\to\{o\}$. Note that the notion of connectivity,
discreteness, and quotient arises in this way. 

\cite{2} gives a longer list 
of notions one can obtain in this way starting from more complicated maps
of finite topological spaces, of up to 7 points. Note compactness arises in this way,
and also contractible, as we saw above.

\begin{lemm} In the category of (all) topological spaces, \begin{itemize}

\item[r:] $(\emptyset\longrightarrow \{o\})^{r}$ is the class of surjections

\item[l:] $(\emptyset\longrightarrow \{o\})^{l}$ is the class of maps $A\longrightarrow B$ where $A\neq \emptyset$ or $A=B$

\item[rr:] $(\emptyset\longrightarrow \{o\})^{rr}=\{\{x\leftrightarrow y\rightarrow c\}\longrightarrow\{x=y=c\}\}^{l}=\{\{x\leftrightarrow y\leftarrow c\}\longrightarrow\{x=y=c\}\}^{l}$ is the class of subsets, i.e. injective maps $A\hookrightarrow B$ where the topology on $A$ is induced from $B$

\item[ lr:] $(\emptyset\longrightarrow \{o\})^{lr}$ is the class of maps $\emptyset\longrightarrow B$, $B$ arbitrary, and $A=B$

\item[ lrr:] $(\emptyset\longrightarrow \{o\})^{lrr}$ is the class of maps $A\longrightarrow B$ which admit a section

\item[ l:] $(\emptyset\longrightarrow \{o\})^{l}$ consists of maps $f:A\longrightarrow B$ such that either $A\neq \emptyset$ or $A=B=\emptyset$

\item[ rl:] $(\emptyset\longrightarrow \{o\})^{rl}$ is the class of maps of form $A\longrightarrow A\sqcup D$ where $D$ is discrete

\item[ rll:] $(\emptyset\longrightarrow \{o\})^{rll}$ is the class of maps $A\to B$ such that each connected subset of $B$ intersects the image of $A$; 
for “nice” spaces it means that the map $\pi_0(A)\to \pi_0(B)$ is surjective, 
where “nice” means that connected components are both open and closed.

\item[ rllr:] $(\emptyset\longrightarrow \{o\})^{rllr}$ is the class of maps of form $A\to A\sqcup B$ where $A\sqcup B$ 
denotes the disconnected union of $A$ and $B$.

\item[lrrr:] $\{\emptyset\longrightarrow\{o\}\}^{lrrr}$ is the class of injective maps, i.e. such that $f(x)\neq f(y)$ whenever $x\neq y$

\item[lrrrr:] $\{\emptyset\longrightarrow\{o\}\}^{lrrrr}$ is the class of maps such that an arbitrary (not necessarily continuous) section is necessarily continuous

\item[lrrrl:] $\{\emptyset\longrightarrow\{o\}\}^{lrrrl}$ is the class of {\em quotients}, i.e.~the maps $f:A\to B$ such that 
a subset $U\subset B$ is open in $B$ iff its preimage $f^{-1}(U)\subset A$ is open in $A$. 
\end{itemize}

\end{lemm}
\begin{proof} Each is an easy exercise in diagram chasing and point set topology.
\end{proof} 

In \verb|lrrrl|, we apply Quillen negation 5 times and get a meaningful notion. 
Can it be more than 5 ? I.e.~can we apply Quillen negation $>5$ times to something simple or natural,
and still get a meaningful and/or well-known notion ? 

\subsection{Acknowledgements.} Tyrone (Cutler?) at \href{https://mathoverflow.net/questions/409266/are-trivial-fibrations-of-finite-cw-complexes-soft-for-normal-maps}{mathoverflow.net} brought Michael selection theory to our attention. 
We thank Martin Bays, Sergei Ivanov, Vladimir Sosnilo, participants of the A.Smirnov seminar,
and Nicolas Cianci, for helpful discussions.
Readability of \cite{2,3} is due to Urs Schreiber. 
We thank Alexandroff St.Petersburg topology seminar for the invitation. 
It is wishful to think that our expression $(\{ \{o\}\lra \{o\rightarrow c\}\}^{r}_{<5})^{lr}$ for compactness
would have not have fit the strict finitist criteria of Nikolai Schanin, teacher of Grigori Mints, and not 
have been considered inherently vague.

\def\FF{{\mathcal F}}
\begin{figure}
\small
\begin{center}

$ (a)\ \xymatrix{ K \ar[r]|{\text{id}} \ar@{->}[d] & K \ar[d] \\ { K\cup_{\FF}\{\text{``x''}\} } \ar[r] \ar@{-->}[ur]|{{\text{``x''}\mapsto x}}& \{\bullet\} }$
$\ \ \ \ (b)\ \xymatrix{ K \ar[r] \ar@{->}[d] & K \ar[d] \\ { K\cup_{\FF}\{\text{``x''}\} } \ar[r] \ar@{-->}[ur]  & \{\bullet\} }$
$\ \ \ \  (c)\ \xymatrix{ A \ar[r] \ar@{->}[d] & A \ar[d] \\ { A\cup_{\FF}\{\text{``x''}\} } \ar[r] \ar@{-->}[ur]  & \{\bullet\} }$

$ (d)\ \xymatrix{ X \ar[r]|{\text{id}} \ar@{->}[d] & X \ar[d]|{f} \\ { X\cup_{\mathfrak U}\{\text{``x''}\} } \ar[r] \ar@{-->}[ur]  & Y  }$
$\ \ \ \ \ (e)\ \xymatrix{ A \ar[r] \ar@{->}[d] & X \ar[d]|{g} \\ { A\cup_{\mathfrak U}\{\text{``x''}\} } \ar[r] \ar@{-->}[ur]  & Y  }$

$ (f)\ \xymatrix{ A \ar[r]|f \ar@{->}[d] & Y \ar[d]|{g} \\ { X } \ar[r] \ar@{-->}[ur]  & \{\bullet\}  }$
$\ \ \ \ \ (g)\ \xymatrix{ A \ar[r]|f \ar@{->}[d] & Y  \ar@{->}[r] & { \{\boldsymbol{B_1}\leftarrow O\rightarrow \boldsymbol{B_2}\} }  \\ { X }  \ar@{-->}[urr]  & { } & { } }$
$\ \ \ \  (h)\ \xymatrix{ A \ar[r] \ar@{->}[d] & { \{\boldsymbol{B_1}\leftarrow O\rightarrow \boldsymbol{B_2}\} } \ar[d] \\ { X } \ar[r] \ar@{-->}[ur]  & {\{\bullet\} }  }$

$ (i)\ \ \xymatrix{ A \ar[r] \ar@{->}[d] & { \{U\} } \ar[d] \\ { X } \ar[r] \ar@{-->}[ur]  & { \{ U \ra U' \} }  }$
$\ \ \ \ \ \ (j)\ \ \xymatrix{ A \ar[r] \ar@{->}[d] & { \{x\llrra y\} } \ar[d] \\ { X } \ar[r] \ar@{-->}[ur]  & \{x=y\}  }$
$\ \ \ \ \ \ (k)\ \ \xymatrix{ A \ar[r] \ar@{->}[d] & { \{o \ra c \} } \ar[d] \\ { X } \ar[r] \ar@{-->}[ur]  & \{o=c\}  }$

$ (l)\ \xymatrix{ {\{o\} } \ar[r] \ar@{->}[d] & X \ar[d] \\ { \{o \ra c \} } \ar[r] \ar@{-->}[ur]  & { Y }  }$

\end{center}
\caption{\label{fig1}\small
These are equivalent reformulations of quasi-compactness of spaces and its generalisation to maps, that of properness of maps. 
 (a) the identity map $K \xra {\id} K$
factors as $K\lra K\cup_{\FF} \{\oo\} \lra K$ 
 (b) this is also equivalent to $K$ being quasi-compact (we no longer require the arrow $K\lra K$ to be identity)
 (c) and in fact quasi-compact spaces are orthogonal to maps associated with ultrafilters 
 (d) $X\xra f  Y$ is proper, i.e. {\sf d) If $\mathfrak U$ is an ultrafilter on $X$ and if $y \in Y$ is a limit point of the ultrafilter
base $f (U)$, then there is a limit point $x$ of $\mathfrak U$ such that $f (x) = y$.}  \href{http://mishap.sdf.org/mints-lifting-property-as-negation/tmp/Bourbaki_General_Topology.djvu}{[Bourbaki, General Topology, I\S10.2,Th.1(d)]}
 (e) this is also equivalent to $X\xra f  Y$ is proper, i.e. this holds for each ultrafilter $\mathfrak U$ on each space $A$ 
 (f)  The mapping \ensuremath{f} has a continuous extension over \ensuremath{X}
 (h)  for every pair $B_1,B_2$ of disjoint closed subsets of \ensuremath{Y} the inverse images $f^{-1}(B_1)$
and $f^{-1}(B_2)$ have disjoint closures in the space $X$
 (i) the image of $A$ is dense in $B$
 (j) the map $A\lra B$ is injective
 (k) the topology on $A$ is induced from $B$
 (l) for $X$ and $Y$ finite, this means that the map $X\lra Y$ is closed, or, equivalently, proper
} \end{figure}


\begin{thebibliography}{10}


\bibitem[1]{1}
\newblock Wikipedia.
\newblock Lifting property.
\newblock \url{https://en.wikipedia.org/wiki/Lifting_property}


\bibitem[2]{2}
\newblock Ncatlab.
\newblock Lifting property.
\newblock \url{https://ncatlab.org/nlab/show/lift}

\bibitem[3]{3}
\newblock Ncatlab.
\newblock Separation axioms in terms of lifting properties.
\newblock \url{https://ncatlab.org/nlab/show/separation+axioms+in+terms+of+lifting+properties}

\bibitem[4]{4}
\newblock Alexandre Grothendieck.
\newblock Esquisse D’un programme. 
\newblock \url{https://webusers.imj-prg.fr/~leila.schneps/grothendieckcircle/EsquisseEng.pdf}, 
\url{https://webusers.imj-prg.fr/~leila.schneps/grothendieckcircle/EsquisseFr.pdf}


\bibitem[5]{5}
\newblock Ernst Michael.
\newblock  Selected Selection Theorems. 
\newblock The American Mathematical Monthly,  (1956). 63(4), 233. doi:10.2307/2310346 
\newblock \url{https://doi.org/10.2307/2310346}

\bibitem[6]{6}
\newblock Ernst Michael.
\newblock Continuous Selections II. 
\newblock  The Annals of Mathematics, 64(3), 562. doi:10.2307/1969603 
\newblock \url{https://doi.org/10.2307/1969603}

\bibitem[7]{7}
\newblock Ernst Michael.
\newblock Continuous Selections I. 
\newblock The Annals of Mathematics, 63(2), 361. doi:10.2307/1969615 
\newblock \url{https://doi.org/10.2307/1969615} 

\bibitem[8]{8}
\newblock Ernst Michael.
\newblock Another note on paracompact spaces. 
\newblock Proceedings of the American Mathematical Society. 8(4). 1957.
\newblock \url{https://www.ams.org/journals/proc/1957-008-04/S0002-9939-1957-0087079-9/S0002-9939-1957-0087079-9.pdf} 



\bibitem[9]{9}
\newblock Eric Wofsey.
\newblock On the algebraic topology of finite spaces.
\newblock 3/9/2008.
\newblock \url{https://www.docdroid.net/zJI3nw2/finite-spaces-wofsey-pdf}



\bibitem[10]{10}
\newblock Takashi Nishimura, 
Goo Ishikawa. 
\newblock Smooth retracts of Euclidean space.
\newblock  Kodai Math. J. 18(2): 260-265 (1995). 
\newblock DOI: 10.2996/kmj/1138043422.




\bibitem[11]{11}
\newblock  Eric K. van Douwen, Roman Pol. 
\newblock  Countable spaces without extension properties. 
\newblock  Bull. Acad. Polon. Sci. Sér. Sci. Math. Astronom. Phys. 25 (1977), no. 10, 987--991.



\bibitem[12]{Bourbaki}
\newblock Nicolas Bourbaki.
\newblock General Topology.
\newblock I\S10.2, Thm.1(d), p.101 (p.106 of file)
\href{http://mishap.sdf.org/mints-lifting-property-as-negation/tmp/Bourbaki_General_Topology.djvu}{General Topology. I\S10.2, Thm.1(d), p.101 (p.106 of file)}

\bibitem[13]{Engelking}
\newblock Ryszard Engelking.
\newblock General Topology. Thm.3.2.1, p.136.


\bibitem[14]{Taimanov}
\newblock A.~D.~Taimanov. 
\newblock On extension of continuous mappings of topological spaces. 
\newblock Mat. Sb. (N.S.), 31(73):2 (1952), 459-463
\newblock \url{www.mathnet.ru/eng/sm5540} 


\bibitem[14]{mintsGE}
\newblock M.Gavrilovich, K.Pimenov.
\newblock A naive diagram-chasing approach to formalisation of tame topology.
\newblock 2018.
\newblock \url{http://mishap.sdf.org/mintsGE.pdf}


\bibitem[15]{DMG}
\newblock Point set topology as diagram chasing computations. Lifting properties as instances of  negation.
\newblock The De Morgan Gazette \ensuremath{5} no.~4 (2014), 23--32,  ISSN 2053-1451.
\newblock \url{http://mishap.sdf.org/mints/mints-lifting-property-as-negation-DMG_5_no_4_2014.pdf}



\end{thebibliography}
\end{document}